\documentclass[a4paper,12pt]{article}
\usepackage[utf8]{inputenc}
\usepackage{amsmath, amsfonts, amsthm, amssymb, enumerate}

\theoremstyle{plain}
\newtheorem{thm}{Theorem}[section]
\newtheorem{lem}[thm]{Lemma}
\newtheorem{prop}[thm]{Proposition}

\newtheorem{remark}[thm]{Remark}

\theoremstyle{definition}

\theoremstyle{remark}

\title{Classification of Rings which admits some special type of Polynomial functions }
\author{Souvik Dey} 
\date{}

\begin{document}
\maketitle
 
 \begin{abstract}
 We know that for a finite field $F$, every function on $F$ can be given by a polynomial with coefficients in $F$. What about the converse? i.e.  if $R$ is a ring (not necessarily commutative or with unity) such that every function on $R$ can be given by a polynomial with coefficients in $R$, can we say $R$ is a finite field ? We show that the answer is yes, and that in fact it is enough to only require that all bijections be given by polynomials. If we allow our rings to have unity, we show that the property that all characteristic functions can be given by polynomials actually characterizes finite fields and if we moreover allow our rings to be commutative, then to characterize finite fields, it is enough that some special characteristic function be given by a polynomial (with co efficients even in some extension ring). Motivated by this, we determine all commutative rings with unity which admits a characteristic function which can be given by some polynomial with co efficients in the ring.  
 \end{abstract}

$\mathbf {INTRODUCTION}$ : 
\newline 

The fact that for a finite field, every function on the field can be given by a polynomial, is an easy and well-known fact. But what about the converse ? i.e. If $R$ is a ring (not necessarily commutative or with unity) such that every function on $R$ can be given by a polynomial with coefficients in $R$ ( we will make this notion explicit shortly, right after this brief introduction), then can we conclude that $R$ is a finite field ? The answer, as we will show, is in-fact yes, and indeed the hypothesis that every function on $R$ can be given by a polynomial with coefficients in $R$ can actually be weakened to only requiring that all bijections are given by polynomials. This is shown in $\mathbf {Proposition}$ $\mathbf {1.2}$. If we allow our ring to have unity, then we show in $\mathbf {Proposition}$ $\mathbf {1.3}$ that it is enough to only require that all characteristic functions on $R$ (i.e. all indicator functions of subsets of $R$) can be given by polynomials with coefficients in $R$. If we further allow our ring to be commutative with unity, then we show in $\mathbf {Proposition}$ $\mathbf {2.1}$, that it is enough to only require that only the characteristic function of the subset of non-zero elements can be given by a polynomial with coefficients in some bigger ring. Motivated by this, we study commutative rings with unity which admits some non-trivial polynomial characteristic function i.e. in which some characteristic function of some proper non-empty subset can be given by a polynomial, and we show in $\mathbf {Proposition}$ $\mathbf {2.7}$ that such rings are exactly those which are local, zero dimensional, with finite residue field and finite index of nilpotency. 
\newline

$\mathbf {The}$ $\mathbf {Results}$  and the $\mathbf {Proofs}$ : 
\newline 

All rings below are assumed to be non-zero i.e. that they contain at least one non-zero element.
\newline

First let us define the notion of polynomial ring for general rings (not necessarily commutative or with unity) as we will use it. Note that for commutative rings with unity , the notion coincides with the usual one.
\newline

$\mathbf {Definition}$ : Let $R$ be a ring ( not necessarily commutative or with unity) . By the polynomial ring on $R$ , denoted by $R[X]$ , we mean the monoid-ring $R[\mathbb N_0] =\{ f : \mathbb N_0 \to R\}$ of finitely supported functions ( i.e. $\{n \in \mathbb N_0 : f(n) \ne 0 \}$ is finite ) from the monoid of non-negative integers $\mathbb N_0$ to $R$ . This is made into a ring in a standard way by defining addition of functions point wise , $(f+g)(n):=f(n)+g(n)$ and multiplication of two functions defined as $(f.g)(n):=\sum _{k+l=n} f(k)g(l) $ . If $f \in R[\mathbb N_0]$ is such that  $f(k)=0 , \forall k >n$ , then we usually write $\sum_{i=0}^n f(i) X^i \in R[X]$ to mean the function $f$ , and conversely by $\sum_{i=0}^na_i X^i \in R[X]=R[\mathbb N_0]$ , we mean the function $f : \mathbb N_0 \to R$ such that $f(i)=a_i, i=0,...,n$ $f(k)=0 , \forall k>n$ . For $f(X) =\sum_{i=0}^n a_i X^i \in R[X]$ , and $r\in R$ , by $f(r)$ we mean the element $a_0+\sum_{i=1}^n a_i r^i \in R$ .  
\newline

$\mathbf {Convention}$ :  Let $R \subseteq S$ be rings (not necessarily commutative or with unity). Let $f: R \to R$ be a function. We say that "$f$ can be given by a polynomial" with co-efficients in $S$ if $\exists g(X) \in S[X]$ such that $f(r)=g(r), \forall r \in R$. When we only say that "$f$ can be given by a polynomial", we will mean by a polynomial with coefficients in $R$
\newline

All the results, facts and notions regarding commutative rings used in the paper can be found in \cite{ABO}. 

\section{}

\begin{lem}

Let $R$ be a finite ring (not necessarily commutative or with unity) then the following are equivalent : 

\begin{enumerate}[i)]
\item For every $r,s\in R\setminus \{0\}$, $\exists f(X) \in R[X]$ (depending on $r$ and $s$) such that $f(0)=0$ and $f(r)=s$ 
\item $R$ is a field. 
\end {enumerate}  
\end{lem}
\begin{proof}

That (ii) implies (i) is obvious since in a finite field, every function can be given by a polynomial. Now to prove (i) implies (ii); since $R$ is assumed to be finite, so to show $R$ is a field, it is enough to show, by Wedderburn's little theorem, $R$ has a unity and is a division ring. But $R$ is finite, so it is enough to show $R$ has no non-zero zero-divisor. So let if possible $uv=0$ for some $u,v \in R \setminus \{0\}$. We claim that $R=Ru=Rv$. Indeed, let $r \in R \setminus \{0\}$, then by the hypothesis of (i) , $\exists f(X) \in R[X]$ such that $f(u)=r$, with $f(0)=0$, so then we can write $f(X)=b_mX^m+...+b_1X$, then $r=f(u)=(b_mu^{m-1}+...+b_1)u \in Ru$. Hence $R=Ru$, and similarly $R=Rv$. Thus $R=Rv=(Ru)v=Ruv=\{0\}$, contradicting our hypothesis. Thus $R$ has no non-zero zero divisor, which is what we wanted to prove. 

\begin{prop}
Let $R$ be a (not necessarily commutative or with unity) ring , then the following conditions on $R$ are equivalent : 
\begin{enumerate}[i)]
\item $R$ is a finite field.
\item Every function from $R$ to $R$ can be given by a polynomial with co efficients in $R$
\item Every bijection on $R$ can be given by a polynomial with co efficients in $R$
\end{enumerate}
\end{prop}

\begin{proof}
That (i) implies (ii) is a well-known standard result. (ii) implies (iii) is immediate. So we prove now (iii) implies (i).
If $R$ were infinite, then the set of all bijections on $R$ has same cardinality as that of $\mathcal P(R)$, the power set of $R$ and $R[X]$ has same cardinality as that of $R$, hence there cannot be any surjection from $R[X]$ onto  the set of all bijections on $R$, so not all bijections on $R$ can be given by polynomials if $R$ is infinite. Thus $R$ is finite. Moreover,if all bijections on $R$ can be given by a polynomial, then for $r,s \in R \setminus \{0\}$ since the transposition $(r \space \space s)$ can also be given by a polynomial, we see that the hypothesis in (i) of Lemma 1.1 is satisfied; and since we already know $R$ is finite, hence by Lemma 1.1, $R$ is a finite field. 
\end{proof}

Now if we allow rings with unity , then we can give another characterization of finite fields, in terms of giving some class of functions by polynomials, as follows.

\begin{prop}
Let $R$ be a (not necessarily commutative) ring with identity such that all characteristic functions of $R$ are given by 
polynomials with coefficients in $R$ . Then $R$ is a finite field.
\end{prop}
\begin{proof}

All characteristic functions of R can be given by a polynomial , so there exist a surjection from $R[X]$ onto the set 

$S:= \{f: R \to \{0,1\}$ : $f$ is a function $ \}$ ; hence $|R[X]|\ge  |S|=|\mathcal P(R)|$ . Now if $R$ were to be infinite , then $|R[X]|=|\cup_{n=1}^\infty R^n|=|R|$ , then we would have $|R| \ge |\mathcal P(R)|$ , contradicting Cantor's power set theorem . Hence $R$ must be finite. So then to prove $R$ is a field , it is enough, by Wedderburn's little theorem, to prove that $R$ is a division ring . To show $R$ is a division ring , let $0 \ne u \in R$ , then the characteristic function of $\{u\}$ in $R$ can be given by a polynomial say $f(X) \in R[X]$ where $f(X)=a_nX^n+...+a_1X+a_0$ , then $f(0)=0$ implies $a_0=0$ , and $1=f(u)=(a_nu^{n-1}+...+a_1)u$ , so for every $ 0 \ne u \in R$ , $\exists u' \in R$ such that $u'.u=1$ , and since $R$ is finite , we also have $u.u'=1$ , hence $R$ is a division ring , which we wanted to prove .  
 
\end{proof}

\section{}

From now on, we will consider commutative rings with unity. 
\newline

\begin{prop}
Let $A \subseteq B$ be commutative rings with unity such that the characteristic function of $A^* (=A \setminus \{0\})$ in $A$ is given by a 
polynomial with coefficients in $B$. Then $A$ is a finite field.
\end{prop}
\begin{proof}
Let $f(X):= b_0+b_1X+b_2X^2+ ... +b_nX^n \in B[X]$ be a polynomial such that $f(0)=0$ and $f(A^*)=\{1\}$. Then $b_0=0$ and 
each nonzero element of $A$ is a unit in $B$. In particular, $A$ is an integral domain. Let $\mathfrak{p}$ be a prime ideal 
in $B$ such that $\mathfrak{p} \cap A=(0)$ so that we have an inclusion $A\hookrightarrow B/\mathfrak{p}$. As 
$\bar f \in B/\mathfrak{p}[X]$ is not a constant, the polynomial $\bar f(X)-1$ can have at most finitely many roots 
in $B/\mathfrak{p}$. But each nonzero element of $A$ is a root of $\bar f(X)-1$. Hence $A$ must be a finite field.
\end{proof}

\end{proof}

\begin{lem} 
Let $R$ be a ring of finite characteristic $n$ and $a,b,c \in R$ such that $a=b+c$ and $c$ is a nilpotent element. Then 
we can choose a positive integer $N$ such that $a^{sN}=b^{sN}$ for all positive integers $s$. 
\end{lem}
\begin{proof}
Let $r$ be the index of nilpotency of $c$ and $n=\prod_ip_i^{e_i}$ be a prime factorization of $n$. For each $i$, we can 
choose a nonnegative integer $\beta_i$ such that $p_i^{\beta_i}$ is the largest power of $p_i$ which divides $(r-1)!$. 
Now if we define $N:=\prod_ip_i^{e_i+\beta_i}$ then for any 
positive integer $s$, 
\[a^{sN}=(b+c)^{sN} = b^{sN}+ sNb^{sN-1}c+ ... +\binom{sN}{r-1}b^{sN-r+1}c^{r-1}+c^r\{...\}\ .\]
For any $0<j<r$, $\binom{sN}{j}= sN(sN-1) \ ...\ (sN-j-1)/j!$. By the choice of $N$, $n.j!$ divides $N$. So the 
characteristic of $R$ divides $\binom{sN}{j}$ for all positive integers $j<r$ and consequently $a^{sN}=b^{sN}$.
\end{proof}

\begin{prop}
Let $(R,\mathfrak{m})$ be a zero dimensional local ring with a finite residue field. Then the following assertions hold 
true.
\begin{enumerate}[(i)]
 \item The group of units $R^*$ is a torsion group.
 \item The exponent of Nil $R$ is finite i.e. there exists a positive integer $r$ such that $a^r=0$ for every $a \in$ Nil $R$ (in such case , we will say that $R$ has finite index of Nilpotency)  if and only if the exponent of $R^*$ is finite.
\end{enumerate}
\end{prop}
\begin{proof}
\begin{enumerate}[(i)]
\item Let $|R/\mathfrak{m}|=n$. If $a\in R$ is a unit then $a^{n-1}-1 \in$ Nil $R$ so that $a^{n-1}=1+c$ for some nilpotent 
element $c \in R$. Now it follows from {\it Lemma 2.2} that $a \in R^*$ has a finite order.
\item First let $R$ have a finite index of nilpotency, say $r$. If $u$ is a unit of $R$ and $|R/\mathfrak{m}|=n$ then 
$u^{n-1}=1+c$ for some $c\in$ Nil $R$. Let $p^e$ be the characteristic of $R$ where char $R/\mathfrak{m}=p$. Let $N$ be a 
positive integer such that $p^e.(r-1)!$ divides $N$. Now arguments given as in {\it Lemma 2.2} shows that $u^{N(n-1)}=1$ 
implying that $R^*$ has a finite exponent.\\
Conversely, let $R^*$ have a finite exponent, say $m$. Then any unit of $R$ satisfies the equation $X^m=1$. Therefore if 
$c$ is any nilpotent element of $R$ then $f(c)=0$ where $f\in R[X]$ is the polynomial defined by $f(X):=(X+1)^m-1$. We can 
write $f$ as a sum of two polynomials $f=g+h$, where each nonzero coefficient of $g$ is a unit and each nonzero coefficient 
of $h$ is a nilpotent respectively. Let $s$ be the least positive integer such the coefficient of $X^s$ in $g$ is nonzero. 
As $R[X]$ has a finite characteristic and $h\in R[X]$ is a nilpotent element, {\it Lemma 2.2} implies that there exists a 
positive integer $N$ such that $f^N=g^N$. Consequently, $g(c)^N=0$ for any nilpotent element $c \in R$. But 
$g(c)^N=c^{sN}u$ for some unit $u \in R^*$, implying that $c^{sN}=0$ and therefore $R$ has a finite index of nilpotency. 
\end{enumerate}
\end{proof}

\begin{lem}
Let $B$ be a commutative ring and $f \in B[X]$ be a polynomial. If $A$ is a subring of $B$ such that $f(A)$ is finite and 
$\bar f \in B/\mathfrak{p}[X]$ is not a constant on $A$ for any minimal prime ideal $\mathfrak{p} \subseteq B$ then $A$ is 
residually finite and $|A/\mathfrak{m}|\leqslant |f(A)|.\text{ deg }f$ for all maximal ideals $\mathfrak{m} \subseteq A$.
\end{lem}
\begin{proof}
Let $\mathfrak{q}$ be a minimal prime ideal of $A$. Then there exists a minimal prime ideal $\mathfrak{p}$ in $B$ such that 
$\mathfrak{p}\cap A=\mathfrak{q}$. If $f(A)=\{b_1, b_2, ... ,b_n\} \subseteq B$ then for the natural inclusion 
$A/\mathfrak{q}\hookrightarrow B/\mathfrak{p}$, we have 
$\bar f(A/\mathfrak{q})\subseteq \{\bar b_1, \bar b_2, ... ,\bar b_n\}\subseteq B/\mathfrak{p}$. Since $\bar f$ is not a 
constant, the polynomials $\bar f(X) - \bar b_i$ can have at most deg $\bar f$ roots in $B/\mathfrak{p}$ for all $i$. Hence 
$A/\mathfrak{q}$ can have at most $|\bar f(A/\mathfrak{q})|.\text{ deg } \bar f$ elements. Since this is true for any 
minimal prime ideal $\mathfrak{q}$ in $A$, we conclude that $A$ is residually finite and 
$|A/\mathfrak{m}|\leqslant |f(A)|.\text{ deg }f$ for all maximal ideals $\mathfrak{m} \subseteq A$.
\end{proof}

\begin{lem}
Let $R$ be a commutative ring and $f\in R[X]$ be a polynomial such that $f(R)$ is finite. Suppose that for any 
maximal ideal $\mathfrak{m}$ of $R$, the image of $f$ in $R/\mathfrak{m}[X]$ is a non-constant function (This is needed. 
Otherwise, consider $X(X+1)$ on $\mathbb{Z}/2\mathbb{Z}\times\mathbb{Z}/3\mathbb{Z}$.). Then $R$ is a zero 
dimensional semi-local ring with finite residue fields and the cardinality of Spec $R$ is at most equal to the number of 
prime factors of $|f(R)|$, counted with multiplicity. In particular, if $R$ admits a nontrivial polynomial characteristic 
function then $R$ must be local.
\iffalse \item Let $(R,\mathfrak{m})$ be a local ring. If $f \in R[X]$ is a polynomial such that $|f(R)|$ is finite and 
$\bar f \in R/\mathfrak{m}[X]$ is not a constant then $R$ has a finite index of nilpotency. In particular, if $R$ admits a 
nontrivial polynomial characteristic function then $R$ has a finite index of nilpotency. 
\end{enumerate}
\fi
\end{lem}
\begin{proof}
Let $\mathfrak{p}$ be a minimal prime ideal of $R$ and $\mathfrak{m}\subseteq R$ be a maximal ideal such that 
$\mathfrak{p} \subseteq \mathfrak{m}$. As the image of $f$ in $R/\mathfrak{m}[X]$ is not a constant function, 
$\bar f \in R/\mathfrak{p}[X]$ is not a constant function either. Since this is true for all minimal prime ideals of $R$, 
it follows from {\it Lemma 2.4} that $R$ is residually finite. In fact the proof of Lemma 2.4 shows that $|R/p|$ is finite for every minimal prime ideal $p$ , hence $R/P$ is finite for every prime ideal $P$ of $R$ , hence $R$ is zero dimensional. \\ 
To prove the assertion about the cardinality of Spec $R$, let us first assume $R$ to be a Noetherian ring. Then since $R$ has dimension zero, it is artinian and as $R$ is
residually finite(i.e. all residue fields of $R$ are finite), so $R$ is finite, then we can write it as a product $R=R_1\times R_2 \times ... \times R_l$ where each $R_i$ is a finite 
local ring. For each $i$, let $f_i$ be the image of $f$ under the natural projection $\pi_i:R \rightarrow R_i$. Then it is 
easy to see that $f(R)=f_1(R_1)\times f_2(R_2)\times ... \times f_l(R_l)$. Also, as the image of $f$ in $R/\mathfrak{m}[X]$ 
is not a constant function for all maximal ideals $\mathfrak{m}$ in $R$, $|f_i(R)|>1$ for all $i$. Therefore it follows 
that the cardinality of $\text{Spec }R$ cannot be more than the number of prime factors of $|f(R)|$, counted with 
multiplicity.\\
Next, let $R$ be a general commutative ring and $n$ be the number of prime factors of $|f(R)|$, counted with 
multiplicity. We want to show that $|\text{Spec }R|\leqslant n$. So, if possible, let us choose $n+1$ maximal ideals 
$\mathfrak{m}_1, \mathfrak{m}_2, ... , \mathfrak{m}_{n+1}$ in $R$. As the image of $f$ in $R/\mathfrak{m}_i[X]$ is not a 
constant function for all $i$, we can choose $2n+2$ elements $a_1,b_1,a_2,b_2, ... ,a_{n+1},b_{n+1}\in R$ such that the 
image of $f(a_i)$ is not equal to the image of $f(b_i)$ in $R/\mathfrak{m}_i$ for all $i$. Also, since 
$\mathfrak{m}_1, \mathfrak{m}_2, ... ,\mathfrak{m}_{n+1}$ are distinct maximal ideals, we can choose $n+1$ elements 
$c_1, c_2, ... ,c_{n+1}$ such that $c_i \in \mathfrak{m}_i \setminus \cup_{j,i\neq j}\mathfrak{m}_j$ for all $i$. Let 
$\Omega$ be a finite subset of $R$ satisfying $f(\Omega)=f(R)$. Now consider the Noetherian subring of $R$, say $R'$, 
generated by $\Omega$, $a_i,b_i,c_i$ for all $i$, and the coefficients of $f$. The ring $R'$, being a subring of a 
residually finite ring $R$, is itself residually finite. Also, by construction, $f(R')=f(R)$, 
$\mathfrak{n}_i:=\mathfrak{m}_i\cap R'$ are distinct maximal ideals of $R'$, and the image of the polynomial 
$f\in R'[X]$ in $R'/\mathfrak{n}_i[X]$ is not a constant function for all $i$. So imitating the proof of the Noetherian 
case we can deduce that the number of maximal ideals in $R'$ cannot be more than $n$, a contradiction. We, therefore, 
conclude that the cardinality of Spec $R$ is at most equal to the number of prime factors of $|f(R)|$, counted with 
multiplicity.
\end{proof}

\begin{prop}
Let $(R,\mathfrak{m})$ be a zero dimensional local ring with a finite residue field. Then the following statements are 
equivalent.
\begin{enumerate}[(i)]
 \item There exists a polynomial $f \in R[X]$ such that $f(R)$ is finite and $\bar f \in R/\mathfrak{m}[X]$ is a 
 non-constant polynomial.
 \item The ring $R$ admits a nontrivial polynomial characteristic function (i.e. the characteristic function of some non-empty proper subset of $R$ can be given by a polynomial) .
\end{enumerate}
In this case, $R$ has a finite index of nilpotency.\\
Conversely, if a zero dimensional local ring $(R,\mathfrak{m})$ with a finite residue field has a finite index of 
nilpotency, then any polynomial $f \in R/\mathfrak{m}[X]$ can be lifted to a polynomial $\tilde f \in R[X]$ satisfying 
$|\tilde f(R)|=|f(R/\mathfrak{m})|$.
\end{prop}
\begin{proof}
First, let $f$ be a polynomial in $R[X]$ such that $f(R)$ is finite and $\bar f \in R/\mathfrak{m}[X]$ is not a constant. 
Let $f(R)=\{r_1, ... ,r_m,r_{m+1}, ... ,r_n\}$, where $\{r_1, ... ,r_m\} \subseteq \mathfrak{m}$ and $r_{m+1}, ... ,r_n$ 
are units. It follows from {\it proposition 2.3(i)} that each $r_i$ has a finite order for $i=m+1, ... ,n$. So we can find 
a positive integer $N$ such that $r_i^N=1$ for all $i=m+1, ... ,n$ and $r_j^N=0$ for all $j\leqslant m$, implying that 
$f^N$ is a nontrivial polynomial characteristic function. The reverse implication is obvious. Next, we want to show that if 
$f \in R[X]$ is a nontrivial polynomial characteristic function then $R$ has a finite index of nilpotency. Replacing $f$ by 
$(f-1)^2$, if required, we may assume that $f(0)=0$. So let $f(X)=a_1X+a_2X^2+ ... +a_nX^n$. We can write $f$ as a sum of 
two nonzero polynomials $f=g+h$, where each nonzero coefficient of $g$ is a unit and each nonzero coefficient of$h$ is a 
nilpotent respectively. Let $s$ be the least positive integer such the coefficient of $X^s$ in $g$ is nonzero. As 
$R/\mathfrak{m}$ is 
finite, $R[X]$ has a finite characteristic. Also, $h\in R[X]$ being a nilpotent element, {\it lemma 2.2} implies that 
there exists a positive integer $N$ such that $f^N=g^N$. Consequently, $g(c)^N=0$ for any nilpotent element $c \in R$. But 
$g(c)^N=c^{sN}u$ for some unit $u \in R^*$, implying that $c^{sN}=0$ and therefore $R$ has a finite index of nilpotency.\\
Conversely, let $(R,\mathfrak{m})$ be a zero dimensional local ring with a finite residue field such that $R$ has a finite 
index of nilpotency, say $e$. Let $f \in R/\mathfrak{m}[X]$ be any polynomial. As $R/\mathfrak{m}$ is finite, we can find 
$\alpha_1,\alpha_2, ... ,\alpha_n, \beta_1, \beta_2, ... ,\beta_n \in R$ such that 
$R/\mathfrak{m}=\{\bar \alpha_1,\bar \alpha_2, ... ,\bar \alpha_n\}$, $f(\bar \alpha_i)=\bar \beta_i$ for all $i$, 
$\bar \alpha_i \neq \bar \alpha_j$ for all $i\neq j$, and $\beta_i=\beta_j$ if and only if $\bar \beta_i=\bar \beta_j$ for 
all $i,j$. We will be through if we can find a polynomial $\tilde f \in R[X]$ which lifts $f$ and satisfies 
$\tilde f (R)=\{\beta_1, \beta_2, ... ,\beta_n\}$. As $R$ has a finite index of nilpotency, by {\it proposition 2.3(ii)}, 
$R^*$ has a finite exponent, say $e'$. If $\tilde f \in R[X]$ is defined by 
$\tilde f(X):=\sum_{i=1}^n\beta_i(\prod_{j\neq i}(X-\alpha_j))^{Ne'}$, where $N$ is a positive integer such that $Ne'$ is 
greater than $e$, then it is easy to see that the polynomial $\tilde f$ satisfies the desired properties.
\end{proof}

\begin{prop}
For a commutative ring with unity $R$ , the characteristic function of some non-empty proper subset of $R$ can be given by a polynomial with co-efficients in $R$ (in short , $R$ admits a non-trivial  polynomial characteristic function ) if and only if $R$ is local , zero dimensional , with finite residue field and finite index of nilpotency. 
\end{prop}
\begin{proof} 

Follows readily from Lemma 2.4 , Lemma 2.5 and Proposition 2.6 

\end{proof}

\begin{remark}

Any nontrivial polynomial characteristic function of a local ring $(R,\mathfrak{m})$ must factor through 
$R/\mathfrak{m}$. So arbitrary characteristic functions on $R$ cannot cannot be given by polynomials. For example, if 
the characteristic function of $R^*$ is given by a polynomial then $R$ is finite. 
\end{remark}

\end{document}